\documentclass{ecai}

\usepackage{color}
\usepackage{soul}
\usepackage{url}
\usepackage[hidelinks]{hyperref}
\usepackage[utf8]{inputenc}
\usepackage{caption}
\usepackage{graphicx}
\usepackage{amsmath}
\usepackage{amssymb}
\usepackage{amsthm}
\usepackage{booktabs}
\usepackage{algorithm}
\usepackage{algorithmic}
\usepackage{xspace}
\usepackage{subcaption}
\usepackage{bm}
\usepackage{mathtools}
\usepackage{float}

\newtheorem{theorem}{Theorem}

\newtheorem{proposition}[theorem]{Proposition}

\newtheorem{definition}{Definition}

\newif\ifpreprint
\preprintfalse

\DeclareMathOperator*{\argmax}{argmax}
\DeclareMathOperator*{\argmin}{argmin}
\newcommand{\innp}[2]{\left\langle #1, #2 \right\rangle}
\newcommand{\norm}[1]{\left\| #1 \right\|}
\newcommand{\vx}{\mathbf{x}}
\newcommand{\vvv}{\mathbf{v}}

\newcommand{\vd}{\mathbf{d}}

\newcommand{\vu}{\mathbf{u}}
\newcommand{\vw}{\mathbf{w}}
\newcommand{\vr}{\mathbf{r}}
\newcommand{\va}{\mathbf{a}}
\newcommand{\vs}{\mathbf{s}}

\newcommand{\vb}{\mathbf{b}}
\newcommand{\vz}{\mathbf{z}}

\newcommand{\R}{\mathbb{R}}
\newcommand{\Z}{\mathbb{Z}}

\newcommand{\Xset}{{\ensuremath{\mathcal{X}}}\xspace}
\newcommand{\Scal}{{\ensuremath{\mathcal{S}}}}

\newcommand{\interval}[1]{ [\![  #1 ]\!]}
\newcommand{\intround}[1]{ \lfloor #1 \rceil }

\begin{document}


\begin{frontmatter}


\paperid{123} 


\title{Efficient Sparse Flow Decomposition Methods for {RNA} Multi-Assembly}


\author[A]{\fnms{Mathieu}~\snm{Besançon}\orcid{0000-0002-6284-3033}\thanks{Corresponding Author. Email: mathieu.besancon@inria.fr.}}

\address[A]{Université Grenoble Alpes, Inria, LIG, CNRS}


\begin{abstract}
Decomposing a flow on a Directed Acyclic Graph (DAG) into a weighted sum of a small number of paths is an essential task in operations research and bioinformatics.
This problem, referred to as Sparse Flow Decomposition (SFD), has gained significant interest, in particular for its application in RNA transcript multi-assembly, the identification of the multiple transcripts corresponding to a given gene and their relative abundance.
Several recent approaches cast SFD variants as integer optimization problems, motivated by the NP-hardness of the formulations they consider. We propose an alternative formulation of SFD as a data fitting problem on the conic hull of the flow polytope.
By reformulating the problem on the flow polytope for compactness and solving it using specific variants of the Frank-Wolfe algorithm,
we obtain a method converging rapidly to the minimizer of the chosen loss function while producing a parsimonious decomposition.
Our approach subsumes previous formulations of SFD with exact and inexact flows and can model different priors on the error distributions.
Computational experiments show that our method outperforms recent integer optimization approaches in runtime, but is also highly competitive in terms of reconstruction of the underlying transcripts, despite not explicitly minimizing the solution cardinality.
\end{abstract}
    
\end{frontmatter}


\section{Introduction}

We consider the problem of decomposing flows into a weighted sum of paths in a Directed Acyclic Graph (DAG) $G = (V,E)$, specifically seeking a small support for the weights.
We refer to the problem as the Sparse Flow Decomposition (SFD) problem and will state various mathematical formulations of SFD from the literature.
We define a \emph{pseudo-flow} as a function $E \rightarrow \R_+$ on the edges of the graph and equivalently, as a vector $\vr \in \R_+^{|E|}$.
A flow is a pseudo-flow respecting a conservation constraint on all nodes except the source $s$ and target $t$, i.e., the sum of flows coming into a node is equal to the sum of the flow going out of it.
This problem has been studied intensively in the last years, in particular thanks to the key application of multi-assembly for RNA transcripts in bioinformatics.
In this application, the DAG corresponds to a given gene \emph{splice graph} in which nodes represent exons (RNA sections encoding information on the gene) to which an artificial source and sink are added, and edges between two exons correspond to reads with one exon following the other.
The goal is to identify \emph{transcripts} which are sequences of exons corresponding to one modality of expression of the gene.
A transcript is equivalent to a path in the graph, and the weight associated to that transcript corresponds to the relative abundance of that transcript to express the gene.
We illustrate the problem setup in Figure~\ref{fig:csetup}.

\begin{figure}[h]
\centering
\begin{subfigure}{0.48\textwidth}
\centering
\includegraphics[width=0.85\textwidth]{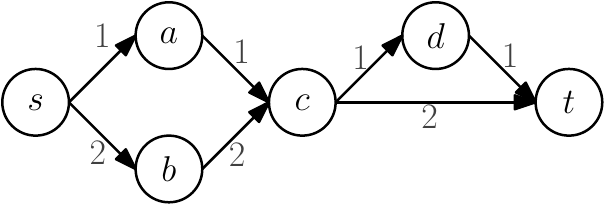}
\caption{Example splice graph with exons $a$ to $d$ and associated flow.\vspace*{0.35cm}}
\label{fig:splicegraph1}
\end{subfigure}
\hfil
\begin{subfigure}{0.48\textwidth}
\centering
\includegraphics[width=0.85\textwidth]{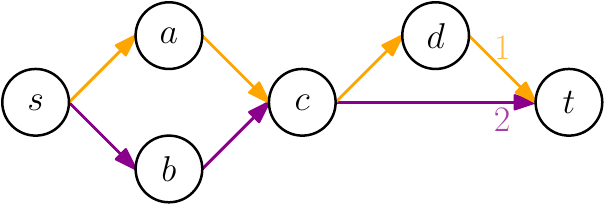}
\caption{A decomposition of the flow into two paths $s-a-c-d-t$ (orange) and $s-b-c-t$ (purple) of weight one and two respectively.\vspace*{0.35cm}}
\label{fig:decomp1}
\end{subfigure}
\begin{subfigure}{0.48\textwidth}
\centering
\includegraphics[width=0.85\textwidth]{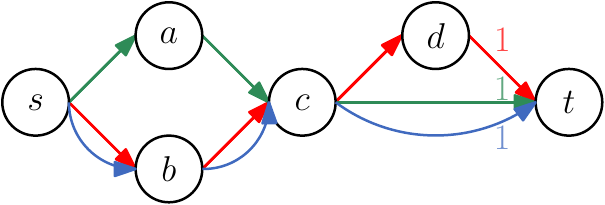}
\caption{A decomposition of the flow into three paths $s-a-c-t$ (green), $s-b-c-t$ (blue), and $s-b-c-d-t$ (red), all of weight one.\vspace*{0.5cm}}
\label{fig:decomp2}
\end{subfigure}
\caption{Illustration of the problem setup. The problem input illustrated in Figure~\ref{fig:splicegraph1} is provided as a directed acyclic graph with flows assigned to edges.
All non-terminal nodes $a\dots d$ correspond to exons. Figure~\ref{fig:decomp1} and \ref{fig:decomp2} show two flow decompositions of the input, producing two disjoint sets of paths.}
\label{fig:csetup}
\end{figure}

The problem has received scrutiny from the bioinformatics, algorithm design, and optimization community in the last decade due to the high relevance for transcriptomics and the growing availability of corresponding data.
We highlight the lines of research connected to our work and more generally SFD and that are compatible with inexact flow models, i.e., when the input pseudo-flow data are contaminated with errors and may not form a flow.

\paragraph*{Related work on Flow Decomposition\\}
Most early work focused on sparse fitting approaches and designing efficient two-step algorithms.
In a first stream of work, convex optimization models designed for sparse regression were leveraged e.g., in \cite{li2011isolasso,li2011sparse} to fit sparse predictive models on the weights associated with paths. One major drawback of such an approach is the need to first enumerate the exponential number of paths before fitting the sparse model.
In order to avoid explicitly working on the exponential number of paths, \cite{tomescu2013novel} design a two-step approach first using a minimum cost flow model to produce a flow fitting the data under a given loss, followed by a heuristic to decompose that flow into a weighted sum of paths.
In \cite{tomescu2015explaining}, the authors consider the same problem with the restriction that only a few paths can be used to decompose the flow.
Indeed, the sparsity in the number of paths corresponds to a property observed on real splice graphs and corresponding transcript abundances.
They also show that several formulations of sparse flow decomposition asking for a given upper bound on the number of allowed paths are NP-hard.
Similarly, the method proposed in \cite{bernard2014efficient} is based on a sparse statistical model for the path weights and encodes the problem as a network flow problem producing a flow which is by design in the convex hull of paths (since the graph is acyclic).
Unlike previous approaches, they model the flow values on the arcs as Poisson random variables and optimize the corresponding log-likelihood instead of the least-square loss.
However, they then require a greedy heuristic to decompose their flow into the corresponding weighted set of paths, which for exact flows is as hard as the original problem, greatly hindering the practicality of the approach.
By casting their formulation and the least-square minimization from \cite{tomescu2013novel} in a form amenable to FW algorithms, we revisit these formulations and show their merit when combined with a sparsity-inducing algorithmic approach instead of a modified formulation.

In another approach to handle errors in the RNA-Seq data, the authors of \cite{williams2019rna} propose a formulation relaxing the constraint that the weighted sum of paths exactly matches a flow
and instead construct an interval of flow values for each edge, and then design a custom heuristic to handle these flow intervals.
One major drawback of this approach is already requiring the production of these intervals, and then seeking a feasible solution in this interval.
Even if the interval contains the underlying flow, it is not given that the minimum-cardinality solution respecting these bounds will fit this ground truth flow.

More recently, the continuous progress in mixed-integer optimization methods and solvers \cite{koch2022progress} allowed considering explicit mixed-integer formulations of sparse flow decomposition problems for instance scales that would not have been tractable a decade ago.
The first integer optimization model was proposed in \cite{dias2022efficient}, also accommodating the inexact version of the problem from \cite{williams2019rna} in which the weighted sum of paths must lie within some distance of the input flow values.
The formulation is based on a quadratic number of variables encoding paths using flow conservation constraints, the flow expressed as the sum of the paths and the weight variables.
The authors also adapt the least-square formulation from \cite{tomescu2015explaining}, resulting in a mixed-integer quadratic (convex) optimization problem which is theoretically hard and computationally harder than mixed-integer linear optimization problems.
As an alternative formulation handling inexact flows, \cite{dias2024accurate} proposed to lift the uncertainty handling from edges to paths in a mixed-integer formulation.
The rationale for handling errors at the path level is that edges appearing in multiple paths are more prone to read errors in the flow value.

An experimental assessment of the minimality assumption was conducted in \cite{kloster2018practical}, showing that even though most ground truth solutions are of minimum support (i.e.~use the smallest possible number of $s-t$ paths), this is not the case for all of them.
This observation naturally leads to seeking \emph{sparsity} of the solution in terms of number of paths, rather than its minimality.
Furthermore, the arguments for computing a set of paths of minimum cardinality do not necessarily hold in the realistic case where the input data are contaminated with errors.
Finally, we note a recent line of work detecting paths that are required in optimal solutions of the decomposition \cite{grigorjew_et_al:LIPIcs.SEA.2024.14,sena2024safe} and that can be exploited in integer optimization formulations.

\paragraph*{Frank-Wolfe algorithms\\}

Frank-Wolfe (FW) or conditional gradient algorithms \cite{frank1956algorithm,levitin1966constrained} optimize differentiable functions over compact convex sets.
They have benefitted from a strong interest in the last decade, in particular thanks to their advantages for large-scale machine learning applications \cite{jaggi2013revisiting},
including their low cost per iteration and possible exploitation of the structure of the constraint set.
At its core, FW produces iterates as convex combinations of a small number of extreme points of the feasible set,
while only requiring that the function is differentiable (and Lipschitz-smooth in typical cases)
and equipped with zeroth and first-order oracles, and that the feasible region can be accessed through a \emph{linear minimization oracle} (LMO),
i.e., an algorithm which, given a direction, computes an extreme point of the feasible region minimizing its inner product with the direction.
On a generic polytope, this LMO can be implemented through linear optimization but on many structured sets, specialized algorithms can implement the LMO without forming the linear problem constraints explicitly.
Finally, we highlight that FW algorithms produce a so-called Frank-Wolfe gap as a by-product at every iteration, which upper-bounds the unknown primal gap.
We refer interested readers to the recent surveys \cite{bomze2021frank,braun2022conditional} for applications and important results on FW algorithms.

\paragraph*{Contributions\\}

Our contributions are the following.
\begin{enumerate}
    \item We propose a formulation of SFD on the flow polytope optimizing either the least-square error or Poisson log-likelihood and a corresponding solution approach based on Frank-Wolfe algorithms. Importantly, the resulting optimization problems are convex and formulated over the convex hull of s-t paths. We provide theoretical justification and computational evidence for sparsity of the solutions obtained by our method.
    \item We establish the convergence rate of our method and evaluate its cost per iteration, showing linear convergence under both the least-square and Poisson log-likelihood losses and despite the lack of strong convexity in both cases.
    \item We evaluate our method compared to recent integer optimization approaches on reference multi-assembly datasets with and without error contamination in the flow data.
 The results show that our approach not only dominates the integer formulations in runtime, but also produces high-quality solutions in terms of reconstruction error, path identification, and solution sparsity.
\end{enumerate}

\paragraph*{Notation and terminology\\}
Vectors are denoted with bold small letters, scalars with standard small letters, and matrices with capital bold letters.
For $n$, $\interval{n} \ \mathrm{:=}\  \{1\dots n\}$. We use $\intround{a}$ for a rounding of $a$ to the closest integer.
$\Delta_n$ denotes the standard simplex. For $u$ a node in the graph, $\delta^{\mathrm{in}}_u$, $\delta^{\mathrm{out}}_u$ will denote the set of edges with that node as destination, origin respectively.
When unspecified, the default norm $\norm{ \cdot }$ is the Euclidean norm in the appropriate vector space. The function $\log(\cdot)$ denotes the natural logarithm.
For a function $f$, we denote with $D^k f[\vu_1\dots\vu_k]$ the $k$-th directional derivative of $f$ along directions $\vu_1\dots \vu_k$.

\section{Sparse Flow Decompositions via Frank-Wolfe Approaches}

In this section, we formulate SFD under the least-square and Poisson models, transforming both into constrained minimization problems over polytopes.

\subsection{Least-square Problem Formulation}\label{sec:prbformulation}

The flow decomposition problem can be viewed as finding a sparse approximation of a given (pseudo-)flow $\vr$ on a DAG given by a conic combination of weighted paths:
\begin{align*}
\min_{\vx, \vvv, \vw}\;& \frac12 \norm{\vx - \vr}^2_2 \\
\text{s.t. } & \sum_{s=1}^{|E| - |V| + 2} w_s \vvv_s = \vx \\
& \vvv_s \in \Xset{}, \norm{\vw}_0 \leq k, w_s \in \Z_+ \forall s \in \interval{k},
\end{align*}
where $\Xset{}$ is the set of $s-t$ paths, and $k$ is an upper bound on the number of paths to use.
If $\vr$ is a flow, then the optimal solution has a zero objective value.
Furthermore, the flow $\vx$ can be decomposed as the weighted sum of at most $|E|- |V| + 2$ paths \cite{vatinlen2008simple}, i.e., the cyclomatic number of the DAG plus one, hence the upper bound on the number of paths.
The path weights $\vw$ represents the abundance of each transcript, the integrality constraints on $\vw$ are a modeling choice depending on prior knowledge on the input data.
The formulation thus seeks the flow $\vx$ that is the closest to $\vr$ in the Euclidean sense and that can be formed as an integer conic combination of $k$ paths.
From a statistical perspective, this modeling choice follows naturally, e.g., from a Gaussian assumption on the errors polluting the flow data.
The formulation captures several previous approaches \cite{tomescu2015explaining,williams2019rna,dias2022efficient} and unifies both exact and inexact flow decompositions.
An aspect of importance for RNA reconstruction and other applications of SFD is producing a sparse decomposition, i.e., a decomposition using a small number of paths.
This consideration motivated several lines of work to formulate the problem objective solely on the solution sparsity, i.e.~minimizing the weight support subject to fitness to the data.
An objective function minimizing the number of paths leads to NP-hard versions of the problem, for which the natural solution method is based on mixed-integer formulations.
Instead, we propose an algorithmic approach to sparsity, leveraging methods based on the Frank-Wolfe algorithm which naturally produce iterates as convex combinations of a small number of vertices.
We reformulate and relax the problem to:
\begin{align*}
\min_{\vx, \vvv, \vw, \tau}\;& \frac12 \norm{\vx - \tau \vr}^2_2 \\
\text{s.t. } & \sum_{k=1}^{s} w_s \vvv_s = \vx\\
& \vvv_s \in \Xset{}, w_s \in \Delta_n \forall s \in \interval{k} \\
& \tau \geq 0.
\end{align*}

The additional $\tau$ variable scales the flow to the appropriate magnitude to be contained in the convex hull of paths instead of its conic hull.
We can then expand its expression by minimizing the objective w.r.t.~$\tau$:
\begin{align*}
& \tau_{\min} \in \argmin_\tau \norm{\vx - \tau \vr}^2 \; \Leftrightarrow \; \tau_{\min} = \frac{\innp{\vr}{\vx}}{\norm{\vr}^{2}},
\end{align*}
leading to our final least-square formulation:
\begin{align*}
\min_{\vx \in \mathrm{conv}(\Xset{})}\;& \frac12 \norm{\vx - \frac{\innp{\vx}{\vr}}{\norm{\vr}^{2}} \vr }^2.
\end{align*}
Importantly, the problem can now be tackled efficiently in a FW setting, since the constraint set admits an LMO implementable through a shortest path computation on the DAG.
One requirement missing from this formulation is that $\vx$ is formed as a \emph{sparse} convex combination of paths.
Instead of enforcing this in the problem formulation, we will ensure sparsity of the iterates and final solution through the algorithm leveraged for the solution process, namely Frank-Wolfe methods and in particular active set-based methods.
The sparsity in the number of vertices used to construct the solutions $\vx_t$ at any iteration $t$ has been a primary motivation for the strong interest in Frank-Wolfe methods in the last decade \cite{jaggi2013revisiting,braun2022conditional}.
For most practical cases, the empirical sparsity is much better than the upper bounds that could be obtained (at most one vertex per iteration for most Frank-Wolfe variants).
More recently, results on active set identification of Frank-Wolfe FW methods provided theoretical evidence and guarantees on the sparsity of some FW methods.
The active set identification property introduced in \cite{bomze2020active} ensures that when using an optimal step size (e.g.~through line search), the Away-step Frank-Wolfe algorithm reaches the optimal face of the problem in a finite number of iterations and never leaves it afterwards.
This property was extended to the Blended Pairwise Conditional Gradient (BPCG) \cite{tsuji2022pairwise} in \cite{wirth2024pivoting}, in both cases, without the need to assume strong convexity.
The result on active set identification comes with a second result which is a non-trivial upper bound on the number of vertices that are required to represent the current iterate.
After the finite number of iterations needed to reach the optimal face $\mathcal{F}^*$, \cite{wirth2024pivoting} establishes a bound of $\mathrm{dim}(\mathcal{F}^*) + 1$ vertices required to form the iterate.
Not only can this bound be reached, but the corresponding convex decomposition can be obtained as the basic solution of an auxiliary linear optimization problem.
Furthermore, the computational evidence pointed out that BPCG was already sparse enough not to require that additional step and provided sufficient sparsity matching this bound.
We present the BPCG algorithm applied to our setting in Algorithm~\ref{alg:algorithm}.
The quantity $g_t$ corresponds to the FW gap, $\texttt{lmo}_G(\cdot)$ computes the shortest path incidence vector with the edge weights as argument.
The function $\mathrm{weight}_{\mathcal{S}}(\vvv)$ returns the weight of the given vertex in the decomposition of the iterate.
\begin{algorithm}[tb]
    \caption{Blended pairwise conditional gradient for flow decomposition}
    \label{alg:algorithm}
    \textbf{Input}: reference flow $\vr$, DAG $G$, loss function $f$\\
    \textbf{Output}: Paths and weights $\{\vvv_k\}_{k \in \interval{|\Scal_T|}}, \{\lambda_k\}_{k \in \interval{|\Scal_T|}}$.
    \vspace*{-0.3cm}
    \begin{algorithmic}[1] 
        \STATE $\vx_0 \gets p_G(\mathbf{1})$, $g_0 \gets +\infty$
        \FOR{$t \in \interval{T}$ }
        \STATE $\vvv_t \gets \texttt{lmo}_G(\nabla f(\vx_t))$
        \STATE $\va_t \gets \argmax_{\vvv \in \Scal_t} \innp{\nabla f(\vx_t)}{\vvv}$\label{line:search1}
        \STATE $\vs_t  \gets \argmin_{\vvv \in \Scal_t} \innp{\nabla f(\vx_t)}{\vvv}$\label{line:search2}
        \STATE $g_t  \gets \innp{\nabla f(\vx_t)}{\vx_t - \vvv_t}$        
        \IF { $\innp{\nabla f(\vx_t)}{\va_t - \vs_t} \geq g_t$ } 
        \STATE $\vd_t \gets \va_t - \vs_t$
        \STATE $\gamma_{\max} \gets \mathrm{weight}_{\Scal_t}(\va_t)$
        \STATE $\gamma_t \gets \argmin_{\gamma \in [0,\gamma_{\max}]} f(\vx_t - \gamma \vd_t)$
        \IF { $\gamma_t < \gamma_{\max}$ }
            \STATE $\Scal_{t+1} \gets \Scal_{t}$
        \ELSE
            \STATE $\Scal_{t+1} \gets \Scal_{t} \backslash \{\va_t\}$
        \ENDIF
        \ELSE
        \STATE $\vd_t \gets \vx_t - \vvv_t$
        \STATE $\gamma_t \gets \argmin_{\gamma \in [0,1]} f(\vx_t - \gamma \vd_t)$
        \IF { $\gamma_t < 1$ }
            \STATE $\Scal_{t+1} \gets \Scal_{t} \cup \{\vvv_t\}$
        \ELSE
            \STATE $\Scal_{t+1} \gets \{\vvv_t\}$
        \ENDIF
        \ENDIF
        \STATE $\vx_{t+1} \gets \vx_t - \gamma_t \vd_t$
        \ENDFOR
        \STATE \textbf{return} $(\vx_T, \Scal_T)$
    \end{algorithmic}
\end{algorithm}

\subsection{Optimal Step Size}

Numerous step-size strategies have been proposed for Frank-Wolfe algorithms, from function-agnostic step sizes based on the iteration count to algorithms efficiently approximating a line search.
For our least-square formulation however, the optimal step size can easily be computed from the problem data as shown in Proposition~\ref{eq:optimalstepsize}.

\begin{proposition}
At any iteration of the BPCG algorithm applied to the least-square problem, given the current iterate $\vx$ and the current direction $\vd$,
the optimal step size $\gamma^*$ is given by:
\begin{equation}\label{eq:optimalstepsize}
\gamma^* = \min \left\{\frac{ \norm{\vr}^2 \innp{\vx}{\vd} - \innp{\vd}{\vr} \innp{\vx}{\vr} }{ \norm{\vd}^2 \norm{\vr}^2 - \innp{\vd}{\vr}^2 }, \gamma_{\max} \right\}.
\end{equation}
\end{proposition}
\begin{proof}
The step size can be derived from the expression of the objective evaluated at $\vx - \gamma \vd$:
\begin{align*}
    f(\vx - \gamma \vd) & = \frac{1}{2} \left\| \vx - \gamma \vd - \frac{\innp{\vx - \gamma \vd}{\vr}}{\norm{\vr}^2} \vr \right\|^2 \\
    & = \frac12 \left\| \vx - \frac{\innp{\vx}{\vr}}{\norm{\vr}^2} \vr - \gamma \left(\vd - \frac{\innp{\vd}{\vr}}{\norm{\vr}^2} \vr\right) \right\|^2 \\
    &= \frac12 \left\| \mathbf{a} - \gamma \mathbf{b} \right\|^2,
\end{align*}
where $\mathbf{a}$, $\mathbf{b}$ are the appropriate expressions substituted here for conciseness.
We can differentiate the loss with respect to $\gamma$, resulting in the unconstrained minimum $\gamma^* = \frac{\innp{\mathbf{a}}{\mathbf{b}}}{\innp{\mathbf{b}}{\mathbf{b}}}$.
By expanding the terms $\mathbf{a}$ and $\mathbf{b}$, we have:
\begin{align*}
    \innp{\mathbf{a}}{\mathbf{b}} &= \innp{\vx}{\vd} - 2 \frac{\innp{\vd}{\vr} \innp{\vx}{\vr}}{\norm{\vr}^2} + \frac{\innp{\vx}{\vr} \innp{\vd}{\vr} \innp{\vr}{\vr}}{\norm{\vr}^4} \\
    \innp{\mathbf{b}}{\mathbf{b}} &= \norm{\vd}^2 - 2 \frac{\innp{\vd}{\vr}^2}{\norm{\vr}^2} + \frac{\innp{\vd}{\vr}^2}{\norm{\vr}^2} \\
    \gamma^* =& \frac{\innp{\mathbf{a}}{\mathbf{b}}}{\innp{\mathbf{b}}{\mathbf{b}}} = \frac{\norm{\vr}^2 \innp{\vx}{\vd} - 2\innp{\vd}{\vr}\innp{\vx}{\vr} + \innp{\vx}{\vr}\innp{\vd}{\vr}}{\norm{\vd}^2 \norm{\vr}^2 - \innp{\vd}{\vr}^2}
\end{align*}
resulting in Equation~\eqref{eq:optimalstepsize} with the appropriate upper bound $\gamma_{\max}$.
\end{proof}

\subsection{Poisson Regression Formulation}

In this section, we revisit the Poisson regression model developed in \cite{bernard2014efficient}.
In this formulation, the flow passing through each exon is modeled as a Poisson random variable with a mean given by the sum of flows passing through that node in the input data.
The log-likelihood minimization problem is expressed as:
\begin{align*}
\min_{\vx \in \mathrm{cone}(\Xset)} \sum_{u \in V} \left[ \sum_{e \in \delta^{\mathrm{in}}_u} x_e - \left(\sum_{e \in \delta^{\mathrm{in}}_u} r_e\right) \log\left(\sum_{e \in \delta^{\mathrm{in}}_u} x_e\right)\right].
\end{align*}
This formulation computes the flow of maximum likelihood, which is a sum of weighted paths with nonnegative weights by the flow decomposition theorem applied to a DAG.
The optimal rescaling technique from the least-square formulation cannot be applied here without losing convexity of the loss function.
In order to formulate an equivalent compact set, we therefore replace $\mathrm{cone}(\Xset)$ with
\begin{align*}
\bar{\Xset} = \{\|\vr\|_{\infty} \vx : \vx \in \Xset \cup \mathbf{0}\},
\end{align*}
the convex hull of paths that can be scaled up to at most the maximum flow on the data.
For a given direction $\mathbf{g}$, the corresponding LMO corresponds to 1) computing the shortest path $\vvv$ with edge lengths given by $\mathbf{g}$ and 2) if $\innp{\vvv}{\mathbf{g}} \leq 0$, returning that vertex, otherwise returning the origin as the minimizer.
This ensures that the formulation is applicable to FW.

\section{Convergence Analysis}

The least-square objective function is not strongly nor strictly convex, its Hessian has one zero eigenvalue.
We will however show it presents a \emph{quadratic growth property} \cite{karimi2016linear}:
\begin{definition}[Quadratic growth property]\label{def:quadgrowth}
Let $f$ be a closed proper convex function defined over a compact \Xset{}, $\mathrm{dist}_*(\cdot)$ the distance to its set of minimizers and $f^*$ its minimal value.
Then it is said to satisfy a quadratic growth property with constant $\mu > 0$ if
\begin{align*}
f(\vx) - f^* \geq \mu\; \mathrm{dist}^2_*(\vx) \; \forall \vx \in \Xset{}.
\end{align*}
\end{definition}

\begin{proposition}\label{prop:objquadratic}
The least-square objective function $f$ respects the quadratic growth condition with $\mu=\frac12$.
\end{proposition}
\begin{proof}
From the expression of the gradient, we can deduce that unconstrained optima are of the form:
\begin{align*}
X^* \ \mathrm{:=}\  \{\vx^* \in \R^{|E|}_+ : \vx^* = \alpha \vr, \alpha \in \R_+\} ,
\end{align*}
which implies that the projection $\mathrm{proj}_*(\vx)$ of a point $\vx$ onto the set of optimizers, and the distance of that point to the set of optimizers $\mathrm{dist}_*(\vx)$ are respectively given by:
\begin{align*}
    & \mathrm{proj}_*(\vx) = \frac{\innp{\vr}{\vx}}{\norm{\vr}^2} \vr, \; \; \; \mathrm{dist}_*(\vx) = \left\| \vx - \frac{\innp{\vr}{\vx}}{\norm{\vr}^2} \vr \right\|.
\end{align*}
The objective function is precisely half of the squared distance, resulting in $\mu=\frac12$ with the notation of Definition~\ref{def:quadgrowth}.
\end{proof}

The BPCG algorithm thus converges linearly for the least-square problem, based on \cite[Theorem 3.6]{wirth2024fast} and Proposition~\ref{prop:objquadratic},
since the quadratic growth condition is a special case of sharpness, and our function is Lipschitz-smooth.

Unlike the least-square formulation, the Poisson objective function is not Lipschitz-smooth, the classic FW convergence results hence do not apply.
However, it is three times differentiable on its domain and self-concordant.
FW with some step-size strategies have been shown to converge on (generalized) self-concordant functions in \cite{dvurechensky2023generalized} and \cite{carderera2024scalable} at the usual $\mathcal{O}(1/t)$ rate and at a linear rate in specific settings (including strongly convex functions) which do not match ours.
Together with the fact that the feasible set is a polytope, \cite{zhao2025new} establishes linear convergence of Away-step Frank-Wolfe without requiring strong convexity if the objective is the composition of a logarithmically homogeneous self-concordant barrier (LHSCB) function (see Definition~\ref{def:logbarrier}) with an affine map.
Their result is extended in \cite{hendrych2023solving} to the BPCG algorithm we apply here, which typically produces sparser iterate than Away-step FW.
We present below the definition of a LHSCB function and apply it to our objective function.
\begin{definition}\label{def:logbarrier}
A convex function $g: \R^n \rightarrow \R$ that is three-times differentiable on its domain is a $\theta$-logarithmically homogeneous self-concordant barrier for a proper (closed, convex, pointed) cone $\mathcal{K}$ iff:
\begin{enumerate}
\item It is self-concordant: $|D^3g(\vz)[\vu,\vu,\vu] | \leq 2 (D^2 g(\vz)[\vu,\vu])^{\frac32}$.
\item It is a barrier for $\mathcal K$: for any sequence $\{\vz_k\}_{k \geq 1}$ such that $\vz_k \rightarrow \mathrm{boundary}(\mathcal K)$, $g(\vz_k) \rightarrow \infty$.
\item It is $\theta$-logarithmically homogeneous: $g(\alpha \vz) = g(\vz) - \theta \log(\alpha)$.
\end{enumerate}
\end{definition}

We show that their framework is applicable to the Poisson objective in Proposition~\ref{prop:objpoisson}.
\begin{proposition}\label{prop:objpoisson}
The Poisson objective function can be written as:
\begin{align}\label{eq:poissonlog}
    f(\vx) = h(A\vx) + \innp{\vb}{\vx},
\end{align}
where $A: \R^{|E|} \rightarrow\R^{|V|}$ is a linear operator, $h: \R^{|V|} \rightarrow \R$ is a logarithmically homogeneous self-concordant barrier function and $\vb \in \R^{|E|}$.
\end{proposition}
\begin{proof}
The expression of the objective as Equation~\eqref{eq:poissonlog} follows from the following elements
\begin{align*}
    & (A \vx)_u = \sum_{e \in \delta^{\mathrm{in}}_u} x_e \\
    & h(\vz) = -\sum_{u\in V} (\sum_{e \in \delta^{\mathrm{in}}_u} r_e) \log z_u \\
    & \innp{\vb}{\vx} = \sum_{u \in V} \sum_{e \in \delta^{\mathrm{in}}_u} x_e.
\end{align*}
The function $h$ is self-concordant as the weighted sum of self-concordant functions with nonnegative weights \cite[Proposition 1]{sun2019generalized}.
It is also $\theta$-logarithmically-homogeneous with
\begin{align*}
    & \theta = (\sum_{e \in \delta^{\mathrm{in}}_u} r_e) \\
    \text{since }\; & h(\alpha\vz) = h(\vz) - \sum_{u \in V} (\sum_{e \in \delta^{\mathrm{in}}_u} r_e) \log(\alpha).
\end{align*}
Finally, $h$ is a log-barrier for $\R^{|V|}_+$ since for any point $\bar\vz \in \mathrm{boundary}(\R^{|V|}_+)$, a sequence $\{\vz_k\}_{k \geq 1}$ of componentwise positive vector such that $\lim\limits_{k \rightarrow \infty} \vz_k = \bar\vz$ results in $\lim\limits_{k\rightarrow +\infty} h(\vz_k) = +\infty$.
\end{proof}
Equipped with the result of Proposition~\ref{prop:objpoisson}, we can use the convergence guarantees of \cite{zhao2025new,hendrych2023solving}
to ensure that the BPCG algorithm we are using achieves a linear convergence rate
\begin{align*}
f(\vx_t) - f^* \leq \mathcal{O}(\exp(-ct)),    
\end{align*}
matching the empirical rate of our computational experiments.


\subsection{Early Termination}\label{sec:earlystop}

First-order methods are known for their good scalability due to a low cost per iteration, although they can be hindered by a high number of iterations compared, e.g., to interior points.
We propose an early stopping criterion for the least-square loss to reduce the number of iterations when the current decomposition is converging towards the scaled flow.
For a given solution $\vx$, we can compute its optimal scaling $\alpha$ by minimizing over $\alpha \geq 0$ the least-square error:
$\norm{\alpha \vx - \vr}^2$. Note that this scales up the solution $\vx$ instead of scaling down the flow ``$\vr$'' in order to obtain an integer solution in the conic hull.
We did not optimize this function directly since it would result in a nonconvex objective.
We can derive the optimal $\alpha^* = \innp{\vx}{\vr} / \norm{\vx}^{2}$ and compute conic weights rounded to the closest integer from the current active set weights $\Scal$: $\left\{\mu_k\right\}_{k \in 1\dots |\Scal|} = \left\{\lfloor \alpha^* \lambda_k \rceil\right\}_{k \in 1\dots |\Scal|}$.
At any iteration, a simple test can be performed to check whether the resulting decomposition exactly matches the original flow $\vr$, in which case we can stop the algorithm.

\subsection{Iteration Cost}

We break down the cost of iterations of Algorithm~\ref{alg:algorithm} from individual components.
A crude upper bound on the cost of the linear minimization oracle is $\mathcal{O}(|V| |E|)$ provided by the Bellman-Ford algorithm performing a single-source shortest path with negative edge lengths.
A finer bound is provided by the Goldberg-Radzik algorithm \cite{goldberg1993heuristic} which obtains a $\mathcal{O}(|V| + |E|)$ runtime with a modification proposed in \cite{cherkassky1996shortest}.
Function and gradient evaluations, as well as the exact step size computation can all be performed in $\mathcal{O}(|E|)$.
The last operation that could dominate the iteration cost is the inner product search over the active set performed in Line~\ref{line:search1} and \ref{line:search2} of Algorithm~\ref{alg:algorithm}.
Indeed, in the worst case, the algorithm would add one vertex to the active set per iteration, resulting in a cost of the active set search $\mathcal{O}(T |E|)$ for the last iterations.
However, this represents a worst-case that is not tight in the light of the active set identification and associated bounds from \cite{bomze2020active,wirth2024pivoting} presented in Subsection~\ref{sec:prbformulation},
with the number of vertices after a finite number of iterations being bounded by the dimension of the optimal face plus one.
For practical purposes of the SFD application, the cost of the active set search has not been limiting in the computational experiments.
Future work could consider applying rerent advancements in inner-product data structures such as the one presented in \cite{xu2021breaking,song2022acceleratingfrankwolfealgorithmusing} to derive a cost of the search almost linear in the number of vertices and dimension.

\section{Computational Experiments}

In this section, we evaluate the performance of our proposed FW approach compared to recent integer-based methods for multi-assembly problems with and without error.
The source code for all experiments is available at \cite{repo} and archived at \cite{besancon_2025_16033781}.

\paragraph*{Experimental Setup\\}
We perform all our computations in Julia 1.11.
The mixed-integer problems are modeled with JuMP 1.23~\cite{lubin2023jump} and solved with SCIP 9.2~\cite{bolusani2024scip}.
We load the DAGs into \texttt{Graphs.jl} and use its implementation of Bellman-Ford for shortest paths.
We compare the solution quality to that of the integer optimization formulations from \cite{dias2022efficient} and \cite{dias2024accurate}.
We use the BPCG implementation present in FrankWolfe.jl 0.4 \cite{besanccon2022frankwolfe,besanccon2025improved}, noted \texttt{FW},
the same algorithm with early termination from Subsection~\ref{sec:earlystop}, noted \texttt{FW-C},
the Poisson loss optimized with FW noted as \texttt{FW-P},
the integer optimization model from \cite{dias2022efficient} noted \texttt{IP}, and the robust integer optimization version noted \texttt{IP-R}.
FW algorithms are limited to 5000 iterations.
All methods are restricted to 1800 seconds as a time limit.
Experiments are performed on a cluster with all nodes equipped with Intel Xeon Gold 6338 2GHz CPUs and 512GB of RAM.

\paragraph*{Instance Data\\}
We use the dataset compiled in \cite{fernando_h_c_dias_2024_10775004} which contains data for four species summarized in Table~\ref{tab:inexactcounts}.
These datasets notably include the ground truth transcripts and their abundance, i.e., the paths and corresponding weights and can thus be used to assess the ability of the various methods to recover the underlying structure.

\begin{table}
\centering
\caption{Statistics on the inexact dataset. Species are indicated by their first letter, human, zebrafish, salmon, mouse.
Non-trivial graphs refer to graphs with an upper bound $k = |E| - |V| + 2$ on the number of paths being greater than one.
\vspace*{0.2cm}
}
\begin{tabular}{lrrrrr}
\toprule
Species & h & z & s & m & total \\
\midrule
\# DAGs  &  11783 & 15664 & 40870 & 13122 & 81439\\
$\%$ non-triv. & 0.45 & 0.29 & 0.36 & 0.36 & 0.36 \\
\bottomrule
\vspace*{0.01cm}
\end{tabular}
\label{tab:inexactcounts}
\end{table}

\subsection{Reconstruction of Inexact Flows}

We apply the FW-based approaches, the integer optimization formulation and the robust integer formulation to the inexact flow instances from \cite{fernando_h_c_dias_2024_10775004}.
We quantify the reconstruction quality with two metrics, the \emph{path error} and \emph{flow error}.
The path error is computed as the number of paths in the solution with a weight different from the ground truth decomposition.
The flow error is the Euclidean distance between the true flow and the reconstructed one.
The results are summarized in Table~\ref{tab:gtstats}.

\setlength{\tabcolsep}{4pt}
\begin{table*}
\centering
\caption{Aggregate statistics on the inexact flow decomposition problem using all non-trivial instances.
The path and flow errors and number of paths are computed on 29328 instances, excluding 33 instances for which either \texttt{IP} or \texttt{IP-R} did not find any primal solution in the time limit.
For the path and flow errors and the number of paths, the first number is the arithmetic mean of the metric, the second is the shifted geometric mean (with shift 1), and the last is the number of instances on which the method performed the best on the given metric.
The relative flow error is computed as the flow error divided by the number of edges of the graph.
The flow error with the (o) label reports the mean (resp.~geometric mean) for instances where both integer models were optimized to proven optimality, thus removing instances for which they computed a solution but could not close the gap.
Only the shifted geometric mean of the runtime is displayed for conciseness.
\vspace*{0.2cm}
}
\begin{tabular}{lrrrrr}
\toprule
Metric & \texttt{FW} & \texttt{FW-C} & \texttt{FW-P} & \texttt{IP} & \texttt{IP-R} \\
\midrule
Path error &  3.64 / 2.23 / 13078  &  3.62 / 2.22 / 13189  &  4.79 / 4.03 / 1846  &  4.38 / 3.68 / 4819  &  2.91 / 1.68 / 16584  \\
Flow error &  4.82e-01 / 1.15e-01 / 27118  &  4.81e-01 / 1.15e-01 / 27125  &  6.27e+01 / 5.55e+00 / 16171  &  4.38e+00 / 3.68e+00 / 104  &  2.91e+00 / 1.68e+00 / 10245  \\
Flow err. (o) &  3.80e-01 / 5.52e-02  &  3.80e-01 / 5.49e-02  &  3.93e+01 / 3.18e+00  &  3.49e+00 / 3.14e+00  &  2.06e+00 / 1.26e+00  \\
Flow rel.~err. &  1.56e-02 / 6.79e-03  &  1.55e-02 / 6.77e-03  &  2.74e+00 / 9.55e-01  &  2.04e-01 / 2.00e-01  &  1.14e-01 / 1.09e-01  \\
\# paths &  4.69 / 3.96 / 10204  &  4.69 / 3.96 / 10204  &  4.79 / 4.03 / 9910  &  3.80 / 3.16 / 24659  &  4.19 / 3.54 / 17682  \\
Time (s) &  2.46e-04  &  7.91e-04  &  3.67e-01  &  3.43e+00  &  7.62e-01  \\
\bottomrule
\end{tabular}
\label{tab:gtstats}
\end{table*}

We observe that despite producing sparser solutions on average, the integer optimization model performs worse than the two FW-based approaches in terms of path and flow error.
The robust optimization model yields the best path error performance at the cost of an increased runtime.
Surprisingly, the robust integer model is not costlier than the original integer model despite the increased number of variables and constraints.
We first analyze the solution qualities on the different metrics. The differences in runtime are analyzed in more details below.
The \texttt{FW-C} method performs slightly better than \texttt{FW} on both error metrics, meaning our specific early termination criterion from Subsection~\ref{sec:earlystop} yields a better solution than reaching the least-square optimum.
Note that this is however dependent on the assumption that the underlying flow is formed from integer weights of the individual paths.
Importantly, we observe that despite producing sparser solutions, the \texttt{IP} model rarely results in the best reconstruction in terms of path and flow errors.
This may imply that least-cardinality is not sufficient as a solution concept to reconstruct transcripts from data, and a statistical approach should be prefer to model the data-generating process, with sparsity being handled through the algorithmic process instead of the formulation.
When considering all instances for which a primal solution was returned, the flow reconstructed by FW methods was superior to \texttt{IP-R}.
The runtime distribution of the different methods are presented in Figure~\ref{fig:timeinexact}.

\begin{figure}[h]
\centering
\includegraphics[width=0.5\textwidth]{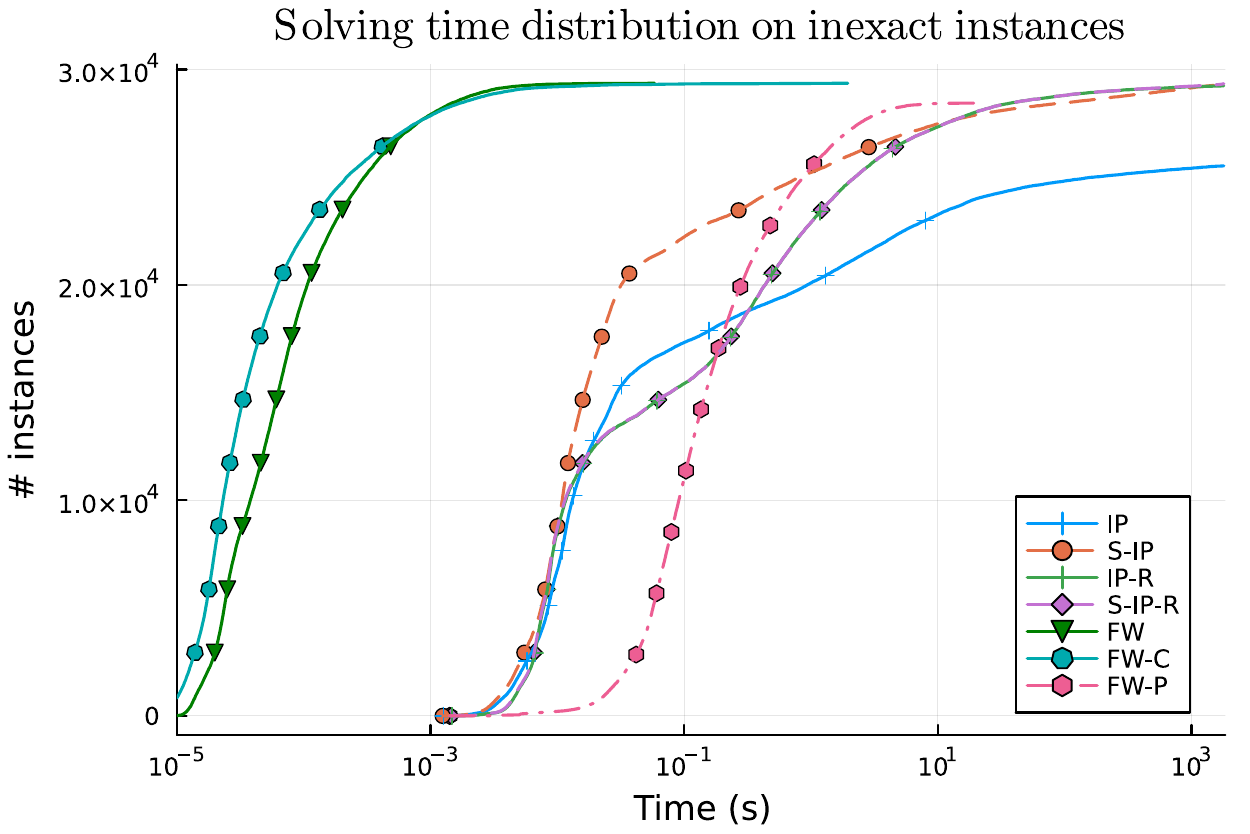}
\caption{Runtime distribution of the different methods on all inexact instances.\vspace*{0.2cm}
\texttt{IP} is the integer optimization formulation with intervals, \texttt{IP-R} is the robust optimization formulation. The prefix \texttt{S-} indicates the time to the best found solution.
\texttt{FW} is the blended pairwise Frank-Wolfe method, \texttt{FW-C} includes the early termination criterion based on the flow integrality assumption.
\texttt{FW-P} is the Poisson regression problem.
}
\label{fig:timeinexact}
\end{figure}

The FW methods applied to the least-square loss clearly outperform all others in runtime.
There are 919 instances on which \texttt{FW-P} reaches the iteration limit due to numerical instabilities which we further analyze below.
The methods \texttt{IP}, \texttt{IP-R} reach the time limit for 3829, resp.~119 instances. The \texttt{FW-C} method terminates faster for small instances than \texttt{FW}, showing that exploiting the weight integrality assumption can accelerate the solution process, in addition to the improved solution quality shown above in Table~\ref{tab:gtstats}.

In Figure~\ref{fig:salmon}, we illustrate the numerical difficulty on an instance created from the inexact dataset with additional Poisson noise on the flow data.
The step size used is the adaptive step size from \cite{pedregosa2020linearly}.
We observe in particular that when using standard 64-bit precision, the primal value and FW gap can stall because of numerical errors.
These numerical errors can be linked to the logarithm terms in the Poisson objective which induces a large range in the magnitude of the coefficients in the gradient.
In addition, the non-smoothness yields challenging subproblems in the adaptive line search, potentially causing excessive estimates of the local Lipschitz constant.
Such numerical challenges with the line search on self-concordant functions were already reported in \cite{carderera2021simple,carderera2024scalable}.
Other similar line searches such as the one introduced in \cite{pokutta2024frank} or \cite{hendrych2025secant} improve stability but at the cost of additional gradient evaluations in the line search procedure.

\begin{figure}
\centering
\includegraphics[trim={0cm 0 0 0.95cm},clip,width=0.52\textwidth]{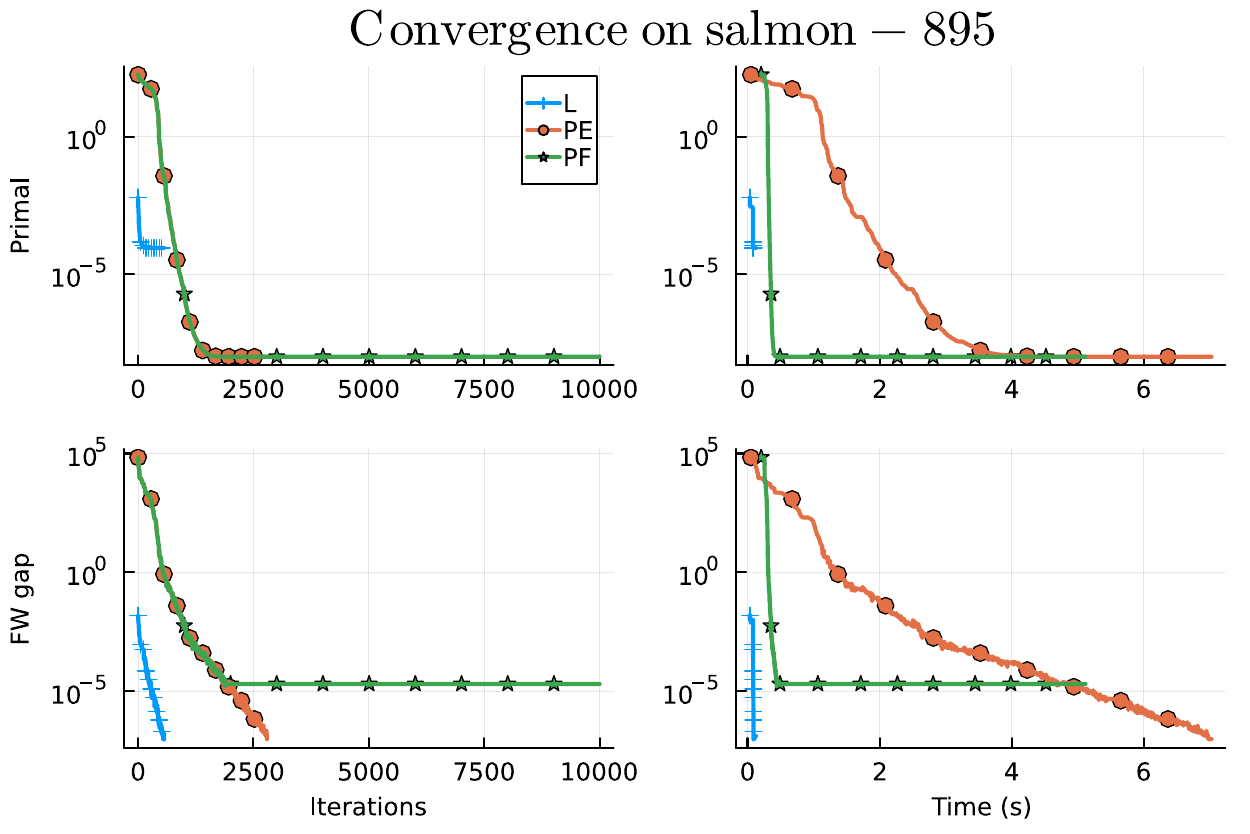}
\caption{
Convergence of the Poisson and least-square formulations on the instance \texttt{salmon-895}, with the primal and FW gap trajectories displayed for the least-square model ``L'', the Poisson model using standard 64-bit floating points ``PF'' and the same loss optimized in extended precision using Julia's \texttt{BigFloat} labeled ``PE''.
\vspace*{0.1cm}
}
\label{fig:salmon}
\end{figure}

\subsection{Reconstruction with Error Distributions}

In order to test the dependence of the different models' performance on the error distribution, we produce more instances from a subset of large splice graphs.
We take all salmon instances for which the upper bound on the number of paths $k$ is at least 16 and the number of edges at least 81, resulting in 109 instances.
On these instances, we perturb the true flow $\bar\vr$ either with a Poisson distribution of parameter $\lambda_{e} = r_e$, or with a binomial distribution of parameters $p=0.5, n=2 r_e$ so that in both cases, the pseudo-flow used by the methods remains nonnegative and of expectation $r_e$.
The results are presented in Table~\ref{tab:errorinstancestats}; we removed \texttt{FW-C} since it performs similarly to \texttt{FW} on both groups of instances and \texttt{IP} given its poor performance on previous instances, and its need for an interval for the value of the flow on each edge.
\begin{table}[b]
\centering
\caption{Results on the 109 large salmon instances with binomial and Poisson distributions of the observed flow.
The presented metrics are identical to Table~\ref{tab:gtstats}.\vspace*{0.1cm}
}
\begin{tabular}{llrrr}
\toprule
Dist.~ & Metric & \texttt{FW} & \texttt{FW-P} & \texttt{IP-R} \\
\midrule
 & Path error &  25.61 / 24.84 / 4  &  20.56 / 19.88 / 51  &  18.86 / 18.40 / 65  \\
 & Flow error &  33.40 / 31.22 / 106  &  408.30 / 291.30 / 0  &  201.71 / 96.57 / 3  \\
Binomial & Flow rel.~err. &  0.36 / 0.35  &  4.47 / 3.33  &  1.97 / 1.24  \\
 & \# paths &  25.53 / 24.76 / 2  &  20.56 / 19.88 / 38  &  17.32 / 16.98 / 77  \\
 & Time (s) &  1.21e-02  &  8.00e+00  &  1.80e+03  \\
\midrule
 & Path error &  26.13 / 25.29 / 3  &  20.54 / 19.90 / 48  &  18.46 / 17.96 / 71  \\
 & Flow error &  49.63 / 45.86 / 107  &  388.32 / 285.85 / 0  &  191.70 / 123.45 / 2  \\
Poisson & Flow rel.~err. &  0.53 / 0.51  &  4.23 / 3.23  &  2.00 / 1.48  \\
 & \# paths &  26.06 / 25.29 / 1  &  20.54 / 19.90 / 33  &  16.91 / 16.43 / 84  \\
 & Time (s) &  2.61e-02  &  7.84e+00  &  1.80e+03  \\
\bottomrule
\end{tabular}
\label{tab:errorinstancestats}
\end{table}

As a first observation, the relative performance and behavior of all three models is similar for the two flow distributions, leading us to conclude that the models are robust beyond their initial modeling assumptions.
We still note that the flow error of \texttt{FW-P} is lower under a Poisson distribution than under a binomial distribution, while the flow error of the least-square model is higher under a Poisson model than under the binomial one.
The effect of the flow distribution on the flow error of \texttt{IP-R} does not follow as clear of a trend, the average relative flow error and the geometric mean of the flow error being higher under a Poisson distribution but the arithmetic mean of the absolute flow error being lower.
The robust integer model \texttt{IP-R} performs the best in a majority of instances in terms of path error, followed by the Poisson regression model \texttt{FW-P}, followed by the least-square model \texttt{FW}.
In terms of flow error, the least-square model outperforms the two other methods by far and achieves the best performance on almost all instances under both flow distributions.

\section{Conclusion}

In this paper, we presented novel formulations for the sparse flow decomposition problem, formulating it as regression problems over the flow polytope and proposed an algorithmic framework producing sparse iterates and corresponding convex decompositions.
The proposed methods are efficient and converge to the optimum of the corresponding formulation at a linear rate, ensuring their applicability to transcriptomics for genes with large splice graphs.
The derived algorithm performs at each iteration a computation of the loss gradient and computes a shortest path on the DAG, thus leveraging specialized low-cost algorithms.
Furthermore, they offer a high reconstruction performance, on par with the best integer optimization model from the literature, showing that even though previous formulations were NP-hard, they are not necessarily the (unique) way to approach transcript multi-assembly.
In particular, these results show that even though the underlying decompositions are sparse, they are not necessarily the \emph{sparsest} and seeking a minimum decomposition does not provide the best reconstruction error or path recovery.
Due to the generic nature of the flow fitting optimization problem, our framework is applicable under a variety of loss functions, leaving the possibility to test more distributional assumptions in future work.
Beyond the problem tackled in this paper, future work will also consider extending our proposed approach to other problems in which one seeks a solution constructed as convex or conic combinations of elementary combinatorial objects, including for instance flow decomposition of non-acyclic graphs \cite{sena2024flowtigs}.




\begin{ack}
We thank Miroslav Kratochvíl, St.~Elmo Wilken and Laurène Pfajfer for insights into the nature of transcriptomics problems and Panayotis Mertikopoulos for fruitful discussions on the objective function of our formulation.
This work benefitted from the support of the FMJH Program PGMO.
\end{ack}

\clearpage
\bibliography{references}

\end{document}